\PassOptionsToPackage{backref=page}{hyperref}
\documentclass[11pt,letterpaper,reqno]{amsart}

\usepackage[T1]{fontenc}
\usepackage{lmodern}
\usepackage{amsmath,amssymb,amsfonts,amsthm}
\usepackage{mathtools}
\usepackage{aliascnt}
\usepackage{microtype}
\usepackage{enumitem}
\usepackage{booktabs}
\usepackage{xcolor}
\usepackage{doi}
\usepackage{hyperref}
\usepackage{tikz}
\usetikzlibrary{calc}
\usetikzlibrary{positioning,arrows.meta}
\usepackage{bookmark}
\usepackage[capitalize,noabbrev]{cleveref}

\hypersetup{
	pdftitle={Obstructions to Jacobi-Finiteness of Quivers with Potentials},
	pdfauthor={Wen Chang and Quanyu Tang},
	pdfsubject={Quivers with potentials, Jacobian algebras, and Golod--Shafarevich--Vinberg inequalities},
	pdfkeywords={quivers with potentials, Jacobian algebras, Jacobi-finiteness, Golod--Shafarevich--Vinberg inequality, completed path algebras},
	pdfstartview={FitH},
	colorlinks=true,
	linkcolor=blue!55!black,
	citecolor=green!35!black,
	urlcolor=blue!55!black
}
\renewcommand*{\backref}[1]{}
\renewcommand*{\backrefalt}[4]{%
	\ifcase #1
	\or
	\quad$\hookleftarrow$, cited on page~#2%
	\else
	\quad$\hookleftarrow$, cited on pages~#2%
	\fi
}

\newtheorem{theorem}{Theorem}[section]

\newaliascnt{lemma}{theorem}
\newtheorem{lemma}[lemma]{Lemma}
\aliascntresetthe{lemma}

\newaliascnt{proposition}{theorem}
\newtheorem{proposition}[proposition]{Proposition}
\aliascntresetthe{proposition}

\newaliascnt{corollary}{theorem}
\newtheorem{corollary}[corollary]{Corollary}
\aliascntresetthe{corollary}

\newaliascnt{conjecture}{theorem}

\aliascntresetthe{conjecture}

\newaliascnt{problem}{theorem}

\aliascntresetthe{problem}

\newaliascnt{question}{theorem}
\newtheorem{question}[question]{Question}
\aliascntresetthe{question}

\theoremstyle{definition}
\newaliascnt{definition}{theorem}

\aliascntresetthe{definition}

\newaliascnt{example}{theorem}
\newtheorem{example}[example]{Example}
\aliascntresetthe{example}

\newaliascnt{remark}{theorem}
\newtheorem{remark}[remark]{Remark}
\aliascntresetthe{remark}

\crefname{theorem}{Theorem}{Theorems}
\Crefname{theorem}{Theorem}{Theorems}
\crefname{lemma}{Lemma}{Lemmas}
\Crefname{lemma}{Lemma}{Lemmas}
\crefname{proposition}{Proposition}{Propositions}
\Crefname{proposition}{Proposition}{Propositions}
\crefname{corollary}{Corollary}{Corollaries}
\Crefname{corollary}{Corollary}{Corollaries}
\crefname{conjecture}{Conjecture}{Conjectures}
\Crefname{conjecture}{Conjecture}{Conjectures}

\DeclareMathOperator{\rad}{rad}
\DeclareMathOperator{\ord}{ord}

\newcommand{\m}{\mathfrak m}
\newcommand{\rA}{\mathfrak r}
\newcommand{\one}{\mathbf 1}
\newcommand{\widehatKQ}{\widehat{KQ}}
\newcommand{\Jac}{\mathcal J}

\title{Obstructions to Jacobi-Finiteness of Quivers with Potentials}

\author[W.~Chang]{Wen Chang}
\address{School of Mathematics and Statistics, Shaanxi Normal University,
	Xi'an 710062, P.~R.~China}
\email{changwen161@163.com}

\author[Q.~Tang]{Quanyu Tang}
\address{School of Mathematical Sciences, University of Science and Technology
	of China, Hefei 230026, P.~R.~China}
\email{tangquanyu827@gmail.com}

\date{}

\begin{document}
	
	\begin{abstract}
		We show that Jacobi-finite potentials need not exist on finite
		$2$-acyclic quivers.
		Our main tool is a matrix-valued
		Golod--Shafarevich--Vinberg inequality for quotients of completed path
		algebras by finitely many, possibly nonhomogeneous, topological relations.
		Applied to cyclic derivatives, it yields a potential-dependent obstruction
		to the finite-dimensionality of completed Jacobian algebras.  We then construct a purely quiver-level criterion excluding every
		Jacobi-finite potential on a given quiver, and exhibit a family of quivers for which every potential has an infinite-dimensional Jacobian algebra. 
	\end{abstract}
	
	\maketitle
	\setcounter{tocdepth}{2} 
	
	\tableofcontents
	\section{Introduction}
	
	Cluster algebras were introduced by Fomin and Zelevinsky in
	\cite{FominZelevinsky}.  In the skew-symmetric setting, the exchange data of
	a seed are encoded by a finite quiver without loops or oriented $2$-cycles,
	and mutation transforms this quiver by a local combinatorial rule. 
	
	For representation-theoretic purposes, however, the quiver alone does not
	record the relations one wants to mutate.  Derksen, Weyman and Zelevinsky
	resolved this by introducing \emph{quivers with potentials}
	\cite{DerksenWeymanZelevinsky}.  A quiver with potential is a pair $(Q,W)$,
	where $W$ is a possibly infinite linear combination of oriented cycles in
	the completed path algebra of the quiver $Q$. The
	cyclic derivatives $\partial_aW$, one for each arrow $a$, generate the
	Jacobian ideal, and their quotient defines the \emph{completed Jacobian algebra}
	$\Jac(Q,W)$. Mutation of quivers with potentials lifts ordinary quiver
	mutation and keeps track of the relations throughout the mutation process. 
	
	Several conditions on a potential play different roles.  A potential is
	\emph{non-degenerate} if every finite sequence of 
	mutations of a quiver with potential has a $2$-acyclic reduced quiver.  A stronger condition is
	\emph{rigidity}: every oriented cycle is cyclically equivalent to an element
	of the Jacobian ideal.  Rigidity implies non-degeneracy
	\cite{DerksenWeymanZelevinsky}.  By contrast, a potential is
	\emph{Jacobi-finite} when $\Jac(Q,W)$ is finite-dimensional.  Over an
	uncountable field, every finite $2$-acyclic quiver admits a non-degenerate
	potential \cite[Corollary~7.4]{DerksenWeymanZelevinsky}; the question studied
	here is whether one can always choose such a potential to be Jacobi-finite as
	well.
	
	Jacobian algebras and quiver-with-potential mutation have become central in
	representation theory.  Amiot used Jacobi-finite quivers with potentials to
	construct Hom-finite generalized cluster categories \cite{Amiot}.
	Ginzburg's differential graded algebra provides a Calabi--Yau enhancement
	\cite{Ginzburg}, and Keller and Yang proved that mutation induces derived
	equivalences between the corresponding differential graded algebras
	\cite{KellerYang}.  The representation type of finite-dimensional Jacobian
	algebras and its behaviour under mutation were studied systematically by
	Gei{\ss}, Labardini-Fragoso and Schr\"oer
	\cite{GeissLabardiniSchroer}.  These developments make the simultaneous
	existence of non-degeneracy and Jacobi-finiteness a natural structural
	problem.

    There are substantial positive results. Non-degenerate Jacobi-finite potentials exist on quivers mutation-equivalent to acyclic quivers \cite{DerksenWeymanZelevinsky,GeissLabardiniSchroer}. They also exist on quivers arising from triangulated surfaces in the sense of Fomin--Shapiro--Thurston \cite{FominShapiroThurston}; see \cite{LabardiniFragoso,LabardiniFragosoIV,LadkaniClosedSurfaces,TrepodeValdivieso,GeuenichLabardiniMiranda}. More generally, the skew-symmetric quivers of finite mutation type were classified by Felikson, Shapiro and Tumarkin \cite{FeliksonShapiroTumarkin}, and the finite-mutation-type case is treated in \cite{HaerizadehYurikusa}. Further positive results are known for a class of quivers introduced by Kontsevich and Soibelman, which is closed under mutation and triangular extensions \cite{LadkaniClassP}, and for quivers associated with Grassmannian cluster algebras \cite{ChangZhang}. In a symplectic-geometric setting, rigid quivers with potentials were
	constructed from Legendrian links \cite{CasalsGao}.

	However, nondegeneracy does not by itself imply Jacobi-finiteness.
	Gei{\ss}, Labardini-Fragoso and Schr\"oer exhibited quivers carrying
	non-degenerate potentials with different finiteness and
	representation-theoretic behaviour \cite{GeissLabardiniSchroer}.  On the mutation-finite 
	exceptional quiver $X_7$, Ladkani constructed explicit non-degenerate
	potentials and showed that Jacobi-finiteness can depend on the characteristic
	of the ground field \cite{LadkaniXSeven}.  Li and Peng obtained
	Jacobi-finite potentials on a new class of $2$-acyclic quivers by a covering
	construction, while leaving the corresponding non-degeneracy assertion as a
	conjecture \cite{LiPeng}.  Accordingly, the generic existence of a
	non-degenerate potential does not by itself control the dimension of the
	Jacobian algebra.  Thus the following question is natural; see Schr\"oer's \emph{Atlas of finite-dimensional algebras} \cite[Question~4.55]{SchroerAtlas}.
	
	\begin{question}
		\label{question:schroer}
		Let $Q$ be a finite $2$-acyclic quiver.  Does there always exist a
		non-degenerate potential $W$ on $Q$ such that $\dim_{\mathbb C}\Jac(Q,W)$ is
		finite?
	\end{question}
	
	We give a negative answer to this question.  Our argument is a completed, multi-vertex form of
	the Golod--Shafarevich method.  Let's recall the two parts of the classical
	picture that are relevant here.
	
	In the scalar one-vertex setting, the original Golod--Shafarevich inequality
	concerns homogeneous relations in a free associative algebra
	\cite{GolodShafarevich,Golod}.  Vinberg extended the method to
	nonhomogeneous relations, measuring a relation by the degree of its lowest
	nonzero homogeneous component.  His formal-power-series version
	treats quotients of algebras of noncommutative formal power series by
	topologically generated ideals \cite{Vinberg}.  
	
	A second line of development retains the vertex idempotents.  Etingof and Eu
	proved a matrix Golod--Shafarevich inequality over a finite-dimensional
	semisimple base for quadratic graded relations
	\cite{EtingofEu}, and Gaddis and Rogalski extended the method
	to homogeneous relations of arbitrary degrees \cite{GaddisRogalski}.
	There are also stronger scalar estimates exploiting the fact that the
	relations are cyclic derivatives of a single potential.  Such refinements
	were developed by Iyudu and Smoktunowicz and by Iyudu and Shkarin
	\cite{IyuduSmoktunowicz,IyuduShkarin}; the latter also treats
	nonhomogeneous potentials and completions.  Brown and Wemyss record both the
	scalar completed Golod--Shafarevich--Vinberg theorem and a
	potential-specific strengthening in \cite{BrownWemyss}.
	
	The setting needed in this paper combines the two directions: we work with a
	completed path algebra with several vertices and possibly nonhomogeneous
	formal relations, while retaining both endpoint idempotents of each
	relation.  The resulting finite-dimensionality criterion is the following.
	The notation appearing in the statement is fixed precisely in
	Section~\ref{sec:matrix-gs}.
	
	\begin{theorem}
		\label{theorem:completed-matrix-gs}
Let $Q$ be a finite quiver with $n$ vertices, and let $M$ be its
adjacency matrix. Let
$A=\widehat{KQ}/\mathcal I$ be defined by finitely many
endpoint-homogeneous relations $r_1,\ldots,r_m$ in
$\mathfrak m^2$.
For $0<\rho<1$, let $R(\rho)$ be
		the matrix recording the $\mathfrak m$-adic orders and the two endpoint
		idempotents of these relations.  If $A$ is finite-dimensional, then there
		exists $H(\rho)\in\mathbb R_{\geq1}^{n}$ such that
		\begin{equation}
			\label{eq:completed-matrix-gs}
			\one\leq
			\bigl({\bf I}_{n}-\rho M^{\top}+R(\rho)\bigr)H(\rho),
		\end{equation}
		where $\bf 1$ is a column vector with entries one, and $\bf I_n$ is the $n\times n$ identity matrix. 
	\end{theorem}

	Applying Theorem~\ref{theorem:completed-matrix-gs} to the cyclic derivatives
	of a potential gives a potential-dependent obstruction, see Theorem \ref{theorem:potential-order}.  We then remove the
	dependence on the potential, see
	Theorem~\ref{theorem:return-obstruction}.
	The concrete consequences proved in Section~\ref{sec:examples} may be
	summarized as follows.
	
	\begin{theorem}
		\label{theorem:main-counterexamples}
		For every field $K$, none of the following quivers admits a Jacobi-finite
		potential:
		\begin{enumerate}[label=\textup{(\roman*)},leftmargin=2.2em]
        \item a $2$-acyclic $d$-regular quiver with $d\geq4$, where $d$-regular means there are $d$ arrows in/out at each vertex, see Proposition~\ref{proposition:regular-quivers};
			\item a quiver $Q_{2,n}$ with $n\geq 10$, where $Q_{2,n}$ is the quiver with vertex set $\mathbb Z/n\mathbb Z$ and arrows $i\longrightarrow i+s$ for each $i\in\mathbb Z/n\mathbb Z$ and $s\in\{1,2\}$, where the labels are read modulo $n$, see Proposition~\ref{proposition:q-n}.\end{enumerate}
	\end{theorem}

	Finally, over an uncountable field every finite $2$-acyclic quiver admits
	a non-degenerate potential
	\cite[Corollary~7.4]{DerksenWeymanZelevinsky}.  In particular,
	the quivers appearing in the above theorem admit non-degenerate potentials over $\mathbb C$, while the theorem says that none of their potentials is Jacobi-finite.
	Thus these quivers give a negative answer to Question~\ref{question:schroer}.	
	
	The paper is organized as follows.  Section~\ref{sec:preliminaries} fixes our
	conventions for completed path algebras and quivers with potentials and
	records a filtered spanning estimate.  Section~\ref{sec:matrix-gs} proves
	Theorem~\ref{theorem:completed-matrix-gs}.  In
	Section~\ref{sec:jacobian} we specialize the inequality first to cyclic
	derivatives and then to directed return distances.  Finally,
	Section~\ref{sec:examples} establishes the two classes of counterexamples
	listed in Theorem~\ref{theorem:main-counterexamples}.
	
	\section{Preliminaries}
	\label{sec:preliminaries}
	
	We begin by fixing the conventions used throughout the paper.  The final subsection records the elementary filtered estimate that will convert a spanning family in the relation module into the matrix inequality of Theorem~\ref{theorem:completed-matrix-gs}.
	
	\subsection{Completed path algebras and conventions}
	
	Let $K$ be a field and let $Q=(Q_0,Q_1,s,t)$ be a finite quiver.  We assume
	throughout that $Q_0=\{1,2,\ldots,n\}$.  For an arrow $a$, its source and
	target are denoted by $s(a)$ and $t(a)$, respectively, and we will denote the arrow by $a: s(a)\rightarrow t(a)$.  Set
	\[
	S=\bigoplus_{i\in Q_0}Ke_i,
	\qquad
	V=\bigoplus_{a\in Q_1}Ka,
	\]
	and regard $V$ as an $S$-bimodule by the convention
	$a\in e_{t(a)}Ve_{s(a)}$.  Thus paths are composed from right to left: if $p=a_d\cdots a_1$, then $a_1$ is traversed first. Denote by $KQ$ the \emph{path algebra} associated to $Q$, where a path from $i$ to $j$ belongs to $e_jKQe_i$.
	
	The \emph{completed path algebra} of $Q$ is
	\[
	\widehatKQ=\prod_{d\geq0}V^{\otimes_S d},
	\]
	where $V^{\otimes_S0}=S$. Its \emph{arrow ideal} is
	$\m=\prod_{d\geq1}V^{\otimes_S d}$, and
	$\m^q=\prod_{d\geq q}V^{\otimes_Sd}$ for $q\geq1$. Thus an element of $\widehatKQ$ is a possibly infinite linear combination of
	paths, and convergence means coefficientwise convergence in increasing path
	length.  All closures and continuity statements below refer to the
	$\m$-adic topology.
	
	For $x\in\widehatKQ$, define its $\m$-adic order by
	$\ord_{\m}(x)=\sup\{q\geq0\mid x\in\m^q\}$, with
	$\ord_{\m}(0)=\infty$.
	We use the convention $\rho^\infty=0$ for $0<\rho<1$.
	
	\subsection{Quivers with potentials}
	
	We next recall the quiver-with-potential terminology needed for the applications in Sections~\ref{sec:jacobian} and~\ref{sec:examples}. We refer the reader to \cite{DerksenWeymanZelevinsky} for details. When discussing quivers with potentials and their mutations, we assume that
	$Q$ has no loops.  Let
	$\operatorname{Pot}(Q)$ be the closed $K$-subspace of $\widehatKQ$ spanned by
	all oriented cycles.  A \emph{potential} on $Q$ is an element
	$W\in\operatorname{Pot}(Q)$.  Potentials are considered up to
	\emph{cyclic equivalence}: two potentials are cyclically equivalent if their
	difference belongs to the closure of the subspace spanned by $a_d\cdots a_2a_1-a_1a_d\cdots a_2$, where $a_d\cdots a_1$ ranges over all oriented cycles.  In other words, the
	choice of a starting point on a cyclic monomial is immaterial.  A
	\emph{quiver with potential}, or QP, is a pair $(Q,W)$.
	
	For an arrow $a\in Q_1$, the \emph{cyclic derivative} $\partial_a:\operatorname{Pot}(Q)\longrightarrow\widehatKQ$ is the unique continuous $K$-linear map determined on a cyclic path
	$p=a_d\cdots a_1$ by
	\[
	\partial_a p
	=\sum_{\substack{1\leq q\leq d\\a_q=a}}
	a_{q-1}\cdots a_1a_d\cdots a_{q+1},
	\]
	where empty products are omitted.  Thus one cuts the cycle immediately after an occurrence of $a$ and then deletes that occurrence.  Since
	$\partial_a(\m^q)\subseteq\m^{q-1}$, cyclic differentiation is continuous.
	If $a:i\to j$, then $\partial_aW\in e_i\widehatKQ e_j$; it is a possibly infinite linear combination of paths from $j$ back to $i$.
	
	The \emph{Jacobian ideal} and the \emph{completed Jacobian algebra} of
	$(Q,W)$ are
	\[
	\mathcal I(W)
	=\overline{\sum_{a\in Q_1}
		\widehatKQ(\partial_aW)\widehatKQ},
	\qquad
	\Jac(Q,W)=\widehatKQ/\mathcal I(W).
	\]
	The closure is essential because both the potential and the relations may be
	infinite formal sums.  The QP $(Q,W)$, or simply the potential $W$, is called
	\emph{Jacobi-finite} if $\Jac(Q,W)$ is finite-dimensional over $K$.
	
	The natural equivalence relation on QPs is right-equivalence.  Suppose that
	$Q$ and $Q'$ have the same vertex set.  A \emph{right-equivalence} from
	$(Q,W)$ to $(Q',W')$ is a continuous $K$-algebra isomorphism $\varphi:\widehat{KQ}\longrightarrow\widehat{KQ'}$ which fixes every vertex idempotent and for which $\varphi(W)$ is cyclically
	equivalent to $W'$.  Right-equivalent QPs have isomorphic completed Jacobian
	algebras.  A QP is called \emph{reduced} if its potential has no quadratic
	part, and \emph{trivial} if its potential is quadratic and its Jacobian
	algebra is $S$.  The splitting theorem of
	\cite{DerksenWeymanZelevinsky} states that every QP is right-equivalent to a
	direct sum of a reduced QP and a trivial QP, and that the right-equivalence
	class of the reduced part is uniquely determined.
	
	For completeness, we also recall how mutation uses this reduction.  Let
	$(Q,W)$ be reduced, and let $k\in Q_0$ be a vertex not incident with an oriented $2$-cycle. For every length-two path $j\xrightarrow{\,b\,}k\xrightarrow{\,a\,}i$, the premutated quiver contains a new arrow $[ab]:j\to i$, and every arrow
	incident with $k$ is replaced by an arrow in the opposite direction.  If
	$[W]$ denotes the potential obtained by replacing each occurrence of a
	subpath $ab$ through $k$ by $[ab]$, then the premutated potential is
	\[
	\widetilde W=[W]+\sum_{a,b}[ab]b^*a^*,
	\]
	where $a^*$ and $b^*$ are the reversed arrows and the sum runs over all such
	pairs $(a,b)$.  The \emph{mutation} $\mu_k(Q,W)$ is, by definition, the reduced part
	of the premutated QP $(\widetilde Q,\widetilde W)$. If the underlying quiver of $\mu_k(Q,W)$ is $2$-acyclic, then it is
the ordinary quiver mutation $\mu_k(Q)$.
	
	A quiver is \emph{$2$-acyclic} if it has neither loops nor oriented
	$2$-cycles.  A potential $W$ on a $2$-acyclic quiver $Q$ is
	\emph{non-degenerate} if, after every finite sequence of QP mutations, the
	underlying reduced quiver remains $2$-acyclic.  Finally, $(Q,W)$ is
	\emph{rigid} if every oriented cycle in $Q$ is cyclically equivalent to an
	element of the Jacobian ideal $\mathcal I(W)$.  Every rigid QP is
	non-degenerate \cite{DerksenWeymanZelevinsky}.  Since every oriented cycle in
	a $2$-acyclic quiver has length at least $3$, one has
	$\partial_aW\in\m^2$ for every arrow $a$.

	\subsection{A filtered spanning estimate}\label{subsec:filtered-spanning-estimate}
	
	The proof of the main matrix inequality will use only the following elementary
	observation about finite filtrations.  We record it separately so that the
	later dimension count remains transparent.
	
	Let
	$X=F^0X\supseteq F^1X\supseteq F^2X\supseteq\cdots$
	be a decreasing filtration of a finite-dimensional $K$-vector space, with
	$F^qX=0$ for all sufficiently large $q$.  For $0<\rho<1$, set
	\[
	h_\rho(X)=\sum_{q\geq0}
	\dim_K(F^qX/F^{q+1}X)\rho^q.
	\]
	
	\begin{lemma}
		\label{lemma:filtered-spanning}
		Suppose that a finite family $(z_\lambda)_{\lambda\in\Lambda}$ spans $X$,
		and that $z_\lambda\in F^{d_\lambda}X$ for integers $d_\lambda\geq0$.
		Then
		\[
		h_\rho(X)\leq\sum_{\lambda\in\Lambda}\rho^{d_\lambda}.
		\]
	\end{lemma}
	
	\begin{proof}
		For every $q\geq0$, the images in $X/F^{q+1}X$ of the vectors with
		$d_\lambda\leq q$ span that quotient.  Hence
		\[
		\dim_K(X/F^{q+1}X)
		\leq
		\#\{\lambda\in\Lambda\mid d_\lambda\leq q\}.
		\]
		If $a_q=\dim_K(F^qX/F^{q+1}X)$, then
		$\dim_K(X/F^{q+1}X)=\sum_{k=0}^q a_k$, and telescoping gives
		\[
		h_\rho(X)
		=(1-\rho)\sum_{q\geq0}\dim_K(X/F^{q+1}X)\rho^q.
		\]
		Therefore
		\[
			h_\rho(X)
			\leq
			(1-\rho)\sum_{q\geq0}
			\#\{\lambda\in\Lambda\mid d_\lambda\leq q\}\rho^q  =\sum_{\lambda\in\Lambda}\rho^{d_\lambda}.
		\qedhere\]
	\end{proof}

	\section{A matrix Golod--Shafarevich inequality for completed path algebras}
	\label{sec:matrix-gs}
	
	We now prove Theorem~\ref{theorem:completed-matrix-gs}.  We first make the notation in its statement precise and then isolate the two structural
	facts needed in the proof.
	
	Let $M=(m_{ij})_{n\times n}$ be the adjacency matrix of $Q$, where
	$m_{ij}$ is the number of arrows from $i$ to $j$.  Vectors are columns
	indexed by $Q_0$.  We write ${\bf I}_{n}$ for the identity $n\times n$ matrix and
	$\one$ for the column vector whose entries are all one.  All vector and
	matrix inequalities are entrywise.
	
	Let $r_1,\ldots,r_m\in\m^2$ be nonzero relations.  We may and do assume that
	each relation is homogeneous with respect to the vertex idempotents.  Indeed,
	an arbitrary relation decomposes as $r_\lambda=\sum_{i,j\in Q_0}e_ir_\lambda e_j$, and replacing the original family by its nonzero corner components does not
	change the closed two-sided ideal it generates.
	
	Thus, for each $1\leq\lambda\leq m$, there are unique vertices $i_\lambda,j_\lambda\in Q_0$ such that $r_\lambda\in e_{i_\lambda}\widehatKQ e_{j_\lambda}$. Equivalently, $r_\lambda$ is a formal linear combination of paths from
	$j_\lambda$ to $i_\lambda$.  Put $d_\lambda=\ord_{\m}(r_\lambda)$, and define
	\[
	\mathcal I_{\mathrm{alg}}
	=\sum_{\lambda=1}^m\widehatKQ r_\lambda\widehatKQ,
	\qquad
	\mathcal I=\overline{\mathcal I_{\mathrm{alg}}},
	\qquad
	A=\widehatKQ/\mathcal I.
	\]
	For $0<\rho<1$, let $R(\rho)$ be the $n\times n$ matrix
	\[
	R(\rho)_{ij}
	=\sum_{\substack{1\leq\lambda\leq m\\
			i_\lambda=i,\ j_\lambda=j}}
	\rho^{d_\lambda}.
	\]
	Thus the entries of $R(\rho)$ record both the orders of the chosen relations
	and their endpoint idempotents.  With this notation, the assertion to be
	proved is exactly Theorem~\ref{theorem:completed-matrix-gs}.
	
	The proof proceeds in three steps.  We first identify the radical filtration
	of $A$, then describe the kernel of the arrow-multiplication map, and finally
	apply Lemma~\ref{lemma:filtered-spanning} to that kernel.
	
	\subsection{The radical of a finite-dimensional quotient}
	
	The following observation is standard.  In the finite-dimensional setting, it is a well-known fact that the image of the arrow ideal coincides with the
	Jacobson radical and is therefore nilpotent.  We include a short proof for
	completeness and to make explicit the role of the closedness of the relation
	ideal in the completed setting.
	
	Let $\pi:\widehatKQ\to A$ be the quotient map and set $\rA=(\m+\mathcal I)/\mathcal I$.
	
	\begin{lemma}
		\label{lemma:radical}
		If $A$ is finite-dimensional, then $\rA=\rad A$ and $\rA$ is nilpotent.
	\end{lemma}
	
	\begin{proof}
		Because every $r_\lambda$ belongs to $\m^2$ and $\m^2$ is closed, one has
		$\mathcal I\subseteq\m^2$.  Hence $A/\rA\cong S$.  For every $q\geq1$, $\rA^q=(\m^q+\mathcal I)/\mathcal I$. Since the ideal $\mathcal I$ is closed, 
		\[
		\bigcap_{q\geq1}\rA^q
		=\left(\bigcap_{q\geq1}(\mathcal I+\m^q)\right)/\mathcal I
		=0.
		\]
		On the other hand, since $A$ is finite-dimensional, the descending chain of powers of $\rA$, as subspaces of $A$, stabilizes.  Its stable value is exactly the above displayed intersection. Therefore there exists $N\geq 1$ such that $\rA^N=0$, that is,
		$\rA$ is nilpotent and $\rA\subseteq\rad A$. The reverse inclusion $\rad A\subseteq\rA$ follows by the semisimplicity of
		$A/\rA\cong S$.
	\end{proof}
	
	\subsection{Relations in the multiplication kernel}
	
	We next identify the relation module that enters the dimension count.  The
	algebraic content is standard: the kernel of the arrow-multiplication map is
	generated by the deconcatenations of the defining relations.  In the
	completed setting one must additionally check that taking the closure of the
	relation ideal introduces no new generators.
	
	Assume in this subsection that $A$ is finite-dimensional.  Set
	\[
	E=V\otimes_SA,
	\qquad
	\mu:E\longrightarrow\rA,
	\qquad
	v\otimes x\longmapsto vx,
	\qquad
	L=\ker\mu.
	\]
	We will compute $L$ as follows.
	For a non-trivial path $p=a_q\cdots a_1$, define left deconcatenation by $\delta(p)=a_q\otimes a_{q-1}\cdots a_1$. It extends to an isomorphism of topological right $\widehatKQ$-modules $\delta:\m\longrightarrow V\widehat\otimes_S\widehatKQ$, where the completed tensor product is the product of the path-length components.  Compose with $\operatorname{id}_V\otimes\pi$ to obtain $\delta_A:\m\longrightarrow E$. Denote by $G_\lambda=\delta_A(r_\lambda)$.
	
	\begin{lemma}
		\label{lemma:kernel-generators}
		For every $\lambda$, one has
		$G_\lambda\in e_{i_\lambda}Le_{j_\lambda}$, and
		\begin{equation*}
			L=\sum_{\lambda=1}^mG_\lambda A.
		\end{equation*}
	\end{lemma}
	
	\begin{proof}
		Every positive-length path has a unique leftmost arrow.  Hence $\delta_A$ is
		surjective, and $\mu\delta_A(y)=\pi(y)$ for any $y\in\m$.
		It follows that $L=\delta_A(\mathcal I)$. Indeed, one inclusion is immediate.  Conversely, if $z\in L$, choose
		$y\in\m$ with $\delta_A(y)=z$.  Then $\pi(y)=\mu(z)=0$, so
		$y\in\mathcal I$.
		
		Since deconcatenation preserves the two outer idempotents, $G_\lambda\in e_{i_\lambda}Ee_{j_\lambda}$. Write $u=u_0+u_+$ with $u_0\in S$ and $u_+\in\m$.  If
		$u_0=\sum_{k\in Q_0}c_ke_k$ and $v\in \widehatKQ$, then left $S$-linearity and right
		$\widehatKQ$-linearity give
		\[
		\delta_A(u_0r_\lambda v)
		=u_0G_\lambda\pi(v)
		=c_{i_\lambda}G_\lambda\pi(v),
		\]
		whereas $\delta_A(u_+r_\lambda v) =\delta_A(u_+)\pi(r_\lambda v)=0$. Conversely, $G_\lambda\pi(v)=\delta_A(r_\lambda v)$.  Therefore
		\[
		\delta_A(\mathcal I_{\mathrm{alg}})
		=\sum_{\lambda=1}^mG_\lambda A.
		\]
		
		By Lemma~\ref{lemma:radical}, choose $N$ with $\rA^N=0$.  Then
		\[
		\delta_A(\m^{N+1})
		\subseteq V\otimes_S\pi(\m^N)
		=V\otimes_S\rA^N
		=0.
		\]
		Thus $\delta_A$ is continuous when the finite-dimensional space $E$ is given
		the discrete topology.  Since every linear subspace of $E$ is closed,
		\[
		\delta_A(\mathcal I)
		=\delta_A\bigl(\overline{\mathcal I_{\mathrm{alg}}}\bigr)
		\subseteq
		\overline{\delta_A(\mathcal I_{\mathrm{alg}})}
		=\delta_A(\mathcal I_{\mathrm{alg}}).
		\]
		The reverse inclusion is immediate, and hence $L=\delta_A(\mathcal I) =\sum_{\lambda=1}^mG_\lambda A$. Finally, $r_\lambda\in\mathcal I$ implies $G_\lambda\in L$.  Together with
		the endpoint condition above, this gives $G_\lambda\in e_{i_\lambda}Le_{j_\lambda}$.
	\end{proof}
	
	\subsection{Proof of the main inequality}
	
	We now combine the preceding description of the multiplication kernel with
	the filtered spanning estimate from Section~\ref{sec:preliminaries}.
	
	\begin{proof}[Proof of Theorem~\ref{theorem:completed-matrix-gs}]
		We use the shifted filtrations
		\[
		F^0E=F^1E=E,
		\qquad
		F^qE=V\otimes_S\rA^{q-1}\quad(q\geq2),
		\]
		and
		\[
		F^0\rA=F^1\rA=\rA,
		\qquad
		F^q\rA=\rA^q\quad(q\geq2).
		\]
		Give $L$ the induced filtration $F^qL=L\cap F^qE$.  Then the multiplication map
		$\mu:E\to\rA$ is strict.  In fact for $q\geq2$,
		\[
		\mu(F^qE)=\pi(V)\rA^{q-1}
		=\pi(V\m^{q-1})=\pi(\m^q)=\rA^q,
		\]
		and strictness for $q=0,1$ is the surjectivity of $\mu$.  Thus the
		associated graded sequence of
		$0\to L\to E\overset{\mu}{\to}\rA\to0$ is exact.
		Furthermore, for each vertex $i$, the induced graded sequence 
$0\to e_iL\to e_iE\overset{\mu}{\to}e_i\rA\to0$ is also exact.

We define
\[
H_i(\rho)=\sum_{q\geq0}
\dim_K(e_i\rA^q/e_i\rA^{q+1})\rho^q,
\qquad \rA^0=A.
\]
The sum is finite by Lemma~\ref{lemma:radical}, and $H_i(\rho)\geq1$
because $A/\rA\cong S$. For $e_i\rA$, the shifted filtration
$F^0\rA=F^1\rA=\rA$ and $F^q\rA=\rA^q$ for $q\geq2$ gives a vanishing
$q=0$ layer, while every subsequent layer matches the corresponding
term of~$H_i(\rho)$.  Since the omitted $q=0$ term of~$H_i(\rho)$ is
$\dim_K(e_iA/e_i\rA)=1$, we obtain
\[h_\rho(e_i\rA)=H_i(\rho)-1,
\]
where $h_\rho$ is defined as in
Subsection~\ref{subsec:filtered-spanning-estimate}.

Since $a\in e_{t(a)}Ve_{s(a)}$, the space $e_iVe_j$ is spanned by the
arrows from $j$ to~$i$, so $\dim_K(e_iVe_j)=m_{ji}$.
Using $1=\sum_{j\in Q_0}e_j$, the $i$-th vertex component of
$E=V\otimes_SA$ decomposes as
\[
 e_iE=\bigoplus_{j\in Q_0}(e_iVe_j)\otimes_K(e_jA).
\]
We now read off $h_\rho(e_iE)$ layer by layer, using the shifted
filtration $F^0E=F^1E=E$ and $F^qE=V\otimes_S\rA^{q-1}$ for $q\geq2$.

The layer $q=0$.
Since $F^0E=F^1E$, the quotient $F^0(e_iE)/F^1(e_iE)$ vanishes.

The layer $q=1$.
We have
\[
 F^1(e_iE)/F^2(e_iE)
 \cong\bigoplus_{j\in Q_0}(e_iVe_j)\otimes_K(e_jA/e_j\rA).
\]
Since 
$e_jA/e_j\rA\cong Ke_j$ is one-dimensional, we have
\[
 \dim_K\bigl(F^1(e_iE)/F^2(e_iE)\bigr)=\sum_{j\in Q_0}m_{ji},
\]
where $\dim_K(e_iVe_j)=m_{ji}$ by our convention.

The layers $q\geq2$.
Similarly,
\[
 F^q(e_iE)/F^{q+1}(e_iE)
 \cong\bigoplus_{j\in Q_0}(e_iVe_j)\otimes_K
       (e_j\rA^{q-1}/e_j\rA^q),
\]
so
\[
 \dim_K\bigl(F^q(e_iE)/F^{q+1}(e_iE)\bigr)
 =\sum_{j\in Q_0}m_{ji}\,
   \dim_K(e_j\rA^{q-1}/e_j\rA^q).
\]

Assembling these dimensions into $h_\rho$ and factoring out
$\rho\sum_jm_{ji}$ yields
\[
 h_\rho(e_iE)
 =\rho\sum_{j\in Q_0}m_{ji}
  \Bigl[\dim_K(e_jA/e_j\rA)
   +\sum_{q\geq1}\rho^{q}\,
    \dim_K(e_j\rA^{q}/e_j\rA^{q+1})
  \Bigr]
 =\rho\sum_{j\in Q_0}m_{ji}\,H_j(\rho).
\]

Then the exactness gives
		\begin{equation}
			\label{eq:exact-weight}
			h_\rho(e_iL)
			=\rho\sum_{j\in Q_0}m_{ji}H_j(\rho)-H_i(\rho)+1.
		\end{equation}
		
		For each $j$, choose a basis $\mathcal B_j$ of $e_jA$ by lifting bases of
		the successive quotients $e_j\rA^q/e_j\rA^{q+1}$.  If
		$b\in\mathcal B_j$ is chosen at level $q$, write $\nu(b)=q$.  Then
		\[
		H_j(\rho)=\sum_{b\in\mathcal B_j}\rho^{\nu(b)}.
		\]
		
		The definition of $d_\lambda$ gives
		$G_\lambda\in F^{d_\lambda}E$.  By
		Lemma~\ref{lemma:kernel-generators}, the vectors $G_\lambda b$ with
		$i_\lambda=i$ and $b\in\mathcal B_{j_\lambda}$ span $e_iL$.
		Such a vector belongs to $F^{d_\lambda+\nu(b)}L$.
		Lemma~\ref{lemma:filtered-spanning} implies
		\begin{equation}
			\label{eq:relation-weight}
			h_\rho(e_iL)
			\leq\sum_{j\in Q_0}R(\rho)_{ij}H_j(\rho).
		\end{equation}
		
		Combining \eqref{eq:exact-weight} and \eqref{eq:relation-weight} gives
		\[
		1\leq H_i(\rho)-\rho\sum_{j\in Q_0}m_{ji}H_j(\rho)
		+\sum_{j\in Q_0}R(\rho)_{ij}H_j(\rho)
		\qquad(i\in Q_0).
		\]
		This is precisely \eqref{eq:completed-matrix-gs}.
	\end{proof}
	
	Before passing to potentials, we compare the theorem with the scalar and graded forms of the Golod--Shafarevich inequality.
	
	\begin{remark}[Relation with classical Golod--Shafarevich--Vinberg inequalities]
		Suppose first that $Q$ has one vertex, let $g=|Q_1|$, and set $f(t)=1-gt+\sum_{\lambda=1}^m t^{d_\lambda}$. Theorem~\ref{theorem:completed-matrix-gs} implies that
		finite-dimensionality of $A$ forces $f(\rho)>0$ for every $0<\rho<1$.
		Thus the existence of $\rho\in(0,1)$ with $f(\rho)\leq0$ gives the
		usual numerical Golod--Shafarevich--Vinberg obstruction \cite{GolodShafarevich,Golod,Vinberg}.
		
		This is not the full statement of Vinberg's formal-power-series theorem.
		Vinberg works with the filtration dimensions of a topological relation space
		and uses a coefficientwise condition on $(1-t)/f(t)$.  In particular, his
		formulation can retain linear dependencies among the initial homogeneous
		parts of the relations, whereas the matrix $R(\rho)$ above is defined from a
		chosen finite generating family \cite{Vinberg}.
	\end{remark}
	
	The inequality immediately yields an infinitude criterion.
	
	\begin{corollary}
		\label{corollary:relation-obstruction}
		Fix $0<\rho<1$.  If there is a vector
		$u\in\mathbb R_{>0}^{n}$ such that
		\[
		u^{\top}
		\bigl({\bf I}_{n}-\rho M^{\top}+R(\rho)\bigr)\leq0,
		\]
		then $A$ is infinite-dimensional.
	\end{corollary}
	
	\begin{proof}
		If $A$ were finite-dimensional, multiply
		\eqref{eq:completed-matrix-gs} on the left by $u^{\top}$.  Since
		$H(\rho)>0$, one would obtain
		\[
		0<u^{\top}\one
		\leq u^{\top}
		\bigl({\bf I}_{n}-\rho M^{\top}+R(\rho)\bigr)H(\rho)
		\leq0,
		\]
		a contradiction.
	\end{proof}
	
	This corollary is the form of the matrix inequality used in the rest of the
	paper.  We now specialize it to Jacobian ideals, whose relations are indexed
	by the arrows of the quiver.
	
	\section{Jacobian algebras and directed distances}
	\label{sec:jacobian}
	
	In this section we apply the completed matrix inequality to cyclic
	derivatives.  We first retain their actual $\m$-adic orders and then bound
	those orders by directed distances in the underlying quiver.
	
	\subsection{Orders of cyclic derivatives}
	
	Let $Q$ be $2$-acyclic and let $W$ be a potential.  For each arrow $a$, put $d_a(W)=\ord_{\m}(\partial_aW)$, and define the matrix $D_W(\rho)$ by
	\[
	(D_W(\rho))_{ij}
	=\sum_{\substack{a\in Q_1\\s(a)=i,\ t(a)=j}}
	\rho^{d_a(W)}.
	\]
	
	Theorem~
	\ref{theorem:completed-matrix-gs} immediately gives the following
	potential-dependent estimate.
	
	\begin{theorem}
		\label{theorem:potential-order}
		If $\Jac(Q,W)$ is finite-dimensional, then for every $0<\rho<1$ there is a
		vector $H(\rho)\in\mathbb R_{\geq1}^{n}$ such that
		\begin{equation}
			\label{eq:potential-order}
			\one\leq
			\bigl({\bf I}_{n}-\rho M^{\top}+D_W(\rho)\bigr)H(\rho).
		\end{equation}
	\end{theorem}
	
	\begin{proof}
		Discard the arrows for which $\partial_aW=0$ and apply
		Theorem~\ref{theorem:completed-matrix-gs} to the remaining cyclic
		relations $r_a=\partial_aW$.  If $a:i\to j$, then
		$r_a\in e_i\widehatKQ e_j$.  With the convention $\rho^\infty=0$, the matrix $R(\rho)$ in
		Theorem~\ref{theorem:completed-matrix-gs} is exactly $D_W(\rho)$.
	\end{proof}
	
	\begin{remark}[The homogeneous-potential specialization]
		Suppose that $W$ is homogeneous of degree $d$.  Every nonzero cyclic
		derivative $\partial_aW$ is then homogeneous of degree $d-1$, and hence $D_W(\rho)\leq \rho^{d-1}M$, with equality when all cyclic derivatives are nonzero.  In the latter case,
		the matrix in \eqref{eq:potential-order} becomes $I-\rho M^{\top}+\rho^{d-1}M$. This is the transpose, arising from our convention for path composition, of
		the untwisted matrix $I-Mt+M^{\top}t^{d-1}$ in \cite[Corollary~3.6]{GaddisRogalski}.  In particular, the matrix in
		Corollary~\ref{corollary:coarse-obstruction} is the same matrix for a
		homogeneous cubic potential.
		
		Potential-specific Golod--Shafarevich estimates may contain an additional
		term of degree $d$, reflecting the syzygy among cyclic derivatives; compare
		\cite{IyuduSmoktunowicz,IyuduShkarin}.  Likewise, under the much stronger
		assumption that the derivation-quotient algebra is twisted Calabi--Yau of
		dimension three, the exact matrix Hilbert-series formula contains the
		additional term $-Pt^d$
		\cite[Theorem~3.1]{GaddisRogalski}.  We do not assume exactness of the
		potential complex or a Calabi--Yau property, so no such higher-syzygy term is
		available in the generality considered here.
	\end{remark}
	
	\begin{corollary}
		\label{corollary:potential-order-obstruction}
		If there are $0<\rho<1$ and $u\in\mathbb R_{>0}^{n}$ such that
		\[
		u^{\top}
		\bigl({\bf I}_{n}-\rho M^{\top}+D_W(\rho)\bigr)\leq0,
		\]
		then $W$ is not Jacobi-finite.
	\end{corollary}
	
	\begin{proof}
		Assume, to the contrary, that $W$ is Jacobi-finite.  Then
		Theorem~\ref{theorem:potential-order} gives a vector $H(\rho)\geq\one$ satisfying
		\eqref{eq:potential-order}.  Multiplication on the left by $u^{\top}$
		yields the required contradiction.
	\end{proof}
	
	\subsection{A quiver-theoretic obstruction}
	
	The preceding criterion still depends on the coefficients of $W$ through
	the orders $d_a(W)$.  We now replace them by quantities determined entirely
	by the directed graph.
	
	For vertices $i,j\in Q_0$, let $\operatorname{dist}_Q(i,j)$ be the length of a shortest directed path from $i$ to $j$, with value $\infty$ when no such path exists. For an arrow $a:i\to j$, set $\ell_Q(a)=\operatorname{dist}_Q(j,i)$. Define the matrix $C_Q(\rho)$ by
	\[
	(C_Q(\rho))_{ij}
	=\sum_{\substack{a\in Q_1\\s(a)=i,\ t(a)=j}}
	\rho^{\ell_Q(a)}.
	\]
	
	\begin{lemma}
		\label{lemma:return-order}
		For every potential $W$ and every arrow $a$, one has
		$d_a(W)\geq\ell_Q(a)$.
	\end{lemma}
	
	\begin{proof}
		Let $a:i\to j$.  Every monomial occurring in $\partial_aW$ is obtained by
		removing an occurrence of $a$ from an oriented cycle and is therefore a
		path from $j$ to $i$.  Each such monomial has length at least
		$\ell_Q(a)$.  Cancellation can only increase the $\m$-adic order.  If there
		is no path from $j$ to $i$, then $a$ lies on no oriented cycle and
		$\partial_aW=0$. By our convention, both $\operatorname{dist}_Q(j,i)$ and $d_a(W)$ are $\infty$ in this case.
	\end{proof}
	
	Combining Lemma~
	\ref{lemma:return-order} with the potential-dependent
	inequality yields the main quiver-theoretic criterion.
	
	\begin{theorem}
		\label{theorem:return-obstruction}
		Suppose that there are $0<\rho<1$ and
		$u\in\mathbb R_{>0}^{n}$ such that
		\begin{equation}
			\label{eq:return-obstruction}
			u^{\top}
			\bigl({\bf I}_{n}-\rho M^{\top}+C_Q(\rho)\bigr)
			\leq0.
		\end{equation}
		Then $\Jac(Q,W)$ is infinite-dimensional for every field $K$ and every
		potential $W$ on $Q$.
	\end{theorem}
	
	\begin{proof}
		Lemma~\ref{lemma:return-order} and $0<\rho<1$ give the entrywise inequality
		$D_W(\rho)\leq C_Q(\rho)$.  Thus \eqref{eq:return-obstruction} implies
		\[
		u^{\top}
		\bigl({\bf I}_{n}-\rho M^{\top}+D_W(\rho)\bigr)\leq0.
		\]
		Corollary~\ref{corollary:potential-order-obstruction} applies.
	\end{proof}
	
	For applications where the individual return distances are inconvenient to
	compute, one may use the uniform lower bound available for every
	$2$-acyclic quiver.  Namely, $\ell_Q(a)\geq2$ for every arrow $a$, and hence
	$C_Q(\rho)\leq\rho^2M$, which gives a coarser but convenient criterion.
	
	\begin{corollary}
		\label{corollary:coarse-obstruction}
		Suppose that $Q$ is $2$-acyclic and that there are $0<\rho<1$ and
		$u\in\mathbb R_{>0}^{n}$ such that
		\begin{equation}
			\label{eq:coarse-obstruction}
			u^{\top}
			\bigl({\bf I}_{n}-\rho M^{\top}+\rho^2M\bigr)
			\leq0.
		\end{equation}
		Then $\Jac(Q,W)$ is infinite-dimensional for every field $K$ and every
		potential $W$ on $Q$.
	\end{corollary}
	
	\begin{proof}
		The absence of loops and oriented $2$-cycles gives
		$\ell_Q(a)\geq2$ for every arrow $a$.  Therefore
		$C_Q(\rho)\leq\rho^2M$, and \eqref{eq:coarse-obstruction} implies
		\eqref{eq:return-obstruction}.
	\end{proof}
	
	The sharper return-distance criterion and its coarser quadratic form will
	now be applied to three families of quivers.

\section{Consequences and examples}
\label{sec:examples}

We now verify the two classes listed in
Theorem~\ref{theorem:main-counterexamples}.  It is useful to begin with the
coarser quadratic criterion.  Although it does not use the precise return
distances, it already produces a large class of counterexamples and, at the
same time, isolates the role played by the local valency of the quiver.  We
then return to the sharper distance-sensitive criterion to obtain a sparser
nine-vertex example, and consider the quiver $Q_{2,n}$ in the second subsection.

\subsection{Regular quivers and a sparse nine-vertex example}

We first record the uniform consequence of
Corollary~\ref{corollary:coarse-obstruction} for regular quivers.

\begin{proposition}
	\label{proposition:regular-quivers}
	Let $Q$ be a finite $2$-acyclic quiver in which every vertex has indegree and
	outdegree $d$.  If $d\geq4$, then $\Jac(Q,W)$ is infinite-dimensional for
	every field $K$ and every potential $W$ on $Q$.
\end{proposition}

\begin{proof}
	Use Corollary~\ref{corollary:coarse-obstruction} with
	$\rho=1/2$ and $u=\one$.  The regularity assumptions give
	\[
	\one^{\top}
	\bigl({\bf I}_{n}-\rho M^{\top}+\rho^2M\bigr)
	=\left(1-\frac d2+\frac d4\right)\one^{\top}
	=\left(1-\frac d4\right)\one^{\top}\leq0.
	\qedhere
	\]
\end{proof}

\begin{remark}
The smallest quiver for which every potential yields an infinite-dimensional Jacobian algebra is as follows.
Let $Q$ be the oriented
	$3$-cycle with four parallel arrows on each edge, as shown in
	Figure~\ref{fig:four-regular-triangle}.  Then $\dim_K\Jac(Q,W)=\infty$
	for every potential $W$.  In terms of the quadratic estimate, this is the
	critical case: at $\rho=1/2$ one has $1-4\rho+4\rho^2=(1-2\rho)^2=0$.
\end{remark}

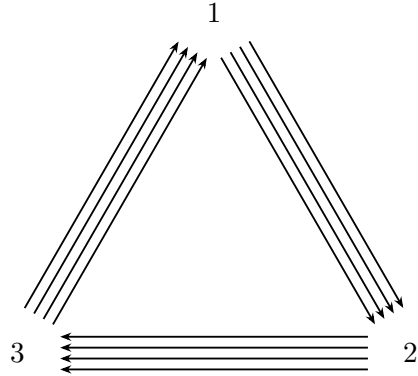
\begin{figure}[htbp]
	\centering
    \begin{tikzpicture}[
	scale=1.2,
	vertex/.style={
		font=\normalsize,
		inner sep=2pt,
		fill=white
	},
	arr/.style={
		-{Stealth[length=1.5mm,width=1.1mm]},
		draw=black,
		line width=0.7pt,
		shorten >=16pt,
		shorten <=16pt
	}
]

\def\r{2.5}

\node[vertex] (v1) at (90:\r) {$1$};
\node[vertex] (v2) at (-30:\r) {$2$};
\node[vertex] (v3) at (210:\r) {$3$};

\foreach \off in {-0.18,-0.06,0.06,0.18} {
	\draw[arr] ($(v1) + \off*(0.866,0.5)$) -- ($(v2) + \off*(0.866,0.5)$);
}

\foreach \off in {-0.18,-0.06,0.06,0.18} {
	\draw[arr] ($(v2) + \off*(0,-1)$) -- ($(v3) + \off*(0,-1)$);
}

\foreach \off in {-0.18,-0.06,0.06,0.18} {
	\draw[arr] ($(v3) + \off*(-0.866,0.5)$) -- ($(v1) + \off*(-0.866,0.5)$);
}

\node[vertex] at (v1) {$1$};
\node[vertex] at (v2) {$2$};
\node[vertex] at (v3) {$3$};

\end{tikzpicture}

	\caption{The oriented 3-cycle with four parallel arrows on each edge admits no potential for which the Jacobian algebra is finite-dimensional.}
	\label{fig:four-regular-triangle}
\end{figure}

\begin{remark}
	The preceding three-vertex example also helps locate the role of our
	inequality relative to existing results.  In the graded setting,
	Etingof and Eu proved a matrix-valued Golod--Shafarevich inequality over
	the semisimple vertex algebra
	$S=\bigoplus_{i\in Q_0}Ke_i$ for quadratic relations
	\cite[Theorem~2.3.4]{EtingofEu}.  Gaddis and Rogalski subsequently gave a
	matrix inequality for homogeneous relations on an arbitrary finite quiver
	\cite[Proposition~3.5]{GaddisRogalski}, with a specialization to
	derivation-quotient algebras defined by homogeneous superpotentials
	\cite[Corollary~3.6]{GaddisRogalski}.  Thus the multi-vertex graded
	situation is already well understood.
	
	On the other hand, the passage from homogeneous relations to
	$\mathfrak m$-adically completed, nonhomogeneous relations is classical
	in the one-vertex setting, beginning with Vinberg's extension of the
	Golod--Shafarevich theorem
	\cite[Theorems~1 and~2]{Vinberg}; see also
	\cite[Theorems~4.6 and~4.7]{BrownWemyss}.  More recently, Iyudu and
	Shkarin developed corresponding estimates adapted specifically to
	potential algebras and their completions \cite{IyuduShkarin}.
	These completed results, however, are formulated for a free algebra, or
	its power-series completion, on a single vertex.
	
	The point needed here is therefore not completion by itself, nor the
	matrix form of the graded Golod--Shafarevich inequality by itself, but
	the simultaneous treatment of both features.  Our matrix inequality works
	in the completed path algebra while retaining the two endpoint
	idempotents of each topological relation.  In particular, closure of the
	relation ideal does not destroy the vertex-by-vertex information needed
	for the deconcatenation argument.  This is what allows
	Proposition~\ref{proposition:regular-quivers} to apply to arbitrary formal
	potentials on the three-vertex quiver above, rather than only to
	homogeneous ones.
\end{remark}

The three-vertex example shows that no lower bound on the number of vertices
can be expected if parallel arrows are allowed.  The situation changes if one
requires the quiver to be simply-laced.  The next elementary observation shows
that, in this case, nine vertices are the first possible order at which the
coarser quadratic criterion can detect a counterexample.

\begin{proposition}
	\label{proposition:coarse-minimality}
	Let $Q$ be a simply-laced $2$-acyclic quiver with $n$ vertices.  If there are
	$0<\rho<1$ and $u\in\mathbb R_{>0}^{n}$ satisfying
	\eqref{eq:coarse-obstruction}, then $n\geq9$.
\end{proposition}

\begin{proof}
	Let $A_Q=M+M^{\top}$.  Because $Q$ is simply-laced and $2$-acyclic,
	$A_Q$ is the adjacency matrix of a simple undirected graph.  From
	\eqref{eq:coarse-obstruction} and $u>0$ one obtains
	\[
	0\geq
	u^{\top}
	\bigl({\bf I}_{n}-\rho M^{\top}+\rho^2M\bigr)u.
	\]
	Since $u^{\top}Mu=u^{\top}M^{\top}u$, the right-hand side is
	\[
	u^{\top}
	\left(
	{\bf I}_{n}-\frac{\rho(1-\rho)}2A_Q
	\right)u.
	\]
	Since $A_Q$ is symmetric, its Euclidean operator norm
	$\lVert A_Q\rVert_2$ equals its spectral radius; this is at most the
	maximum degree and hence at most $n-1$.  Also
	$\rho(1-\rho)\leq1/4$.  If $n\leq8$, then
	\[
	\frac{\rho(1-\rho)}2\,\lVert A_Q\rVert_2
	\leq\frac{n-1}{8}<1,
	\]
	so the displayed quadratic form is positive definite, a contradiction.
\end{proof}

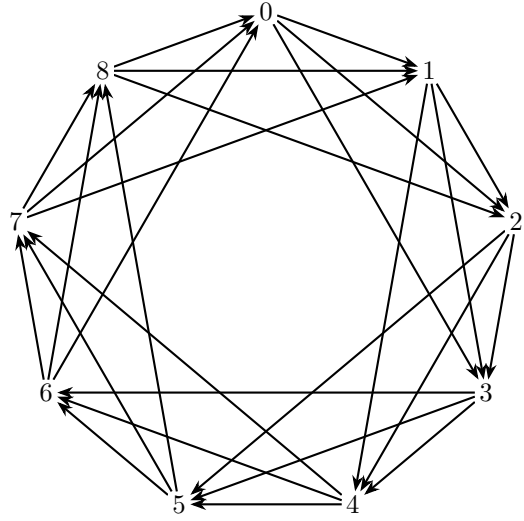
\begin{figure}[htbp]
	\centering
	\begin{tikzpicture}[
		scale=1.2,
		vertex/.style={
			font=\small,
			inner sep=1.5pt,
			fill=white
		},
		arr/.style={
			-{Stealth[length=2.1mm,width=1.5mm]},
			draw=black,
			line width=0.8pt
		}
		]
		
		\def\r{2.8}
		
		\foreach \i in {0,...,8} {
			\node[vertex] (v\i) at ({90-40*\i}:\r) {$\i$};
		}
		
		\foreach \i in {0,...,8} {
			\pgfmathtruncatemacro{\j}{mod(\i+3,9)}
			\draw[arr] (v\i) -- (v\j);
		}
		
		\foreach \i in {0,...,8} {
			\pgfmathtruncatemacro{\j}{mod(\i+2,9)}
			\draw[arr] (v\i) -- (v\j);
		}
		
		\foreach \i in {0,...,8} {
			\pgfmathtruncatemacro{\j}{mod(\i+1,9)}
			\draw[arr] (v\i) -- (v\j);
		}
		
		\foreach \i in {0,...,8} {
			\node[vertex] at (v\i) {$\i$};
		}
		
	\end{tikzpicture}
	\caption{The quiver $Q_{3,9}$, where for every vertex $i\in\mathbb Z/9\mathbb Z$,
		there are arrows from $i$ to $i+1$, $i+2$, and $i+3$.}
	\label{fig:q-nine}
\end{figure}

\begin{example}

The sharper return-distance criterion detects the quiver
$Q=Q_{3,9}$ depicted in Figure~\ref{fig:q-nine}, which has vertex set
$\mathbb Z/9\mathbb Z$ and arrows $i\longrightarrow i+s$ for each
$i\in\mathbb Z/9\mathbb Z$ and $s\in\{1,2,3\}$.
Indeed, $Q_{3,9}$ is $2$-acyclic and, for arrows of steps $1$, $2$,
and $3$, the values of $\ell_Q$ are respectively $3$, $3$, and $2$.

Now take $\rho=1/2$ and $u=\one$. Translation invariance gives $\one^{\top}M^{\top}=3\one^{\top}$, while the three return distances give
	\[
	\one^{\top}C_Q(1/2)
	=
	\left(
	\left(\frac12\right)^3
	+\left(\frac12\right)^3
	+\left(\frac12\right)^2
	\right)\one^{\top}
	=
	\frac12\one^{\top}.
	\]
	Consequently,
	\[
	\one^{\top}
	\bigl(
	{\bf I}_{9}-\tfrac12M^{\top}+C_Q(1/2)
	\bigr)
	=0.
	\]
	Theorem~\ref{theorem:return-obstruction} therefore applies and shows that
	$\Jac(Q_{3,9},W)$ is infinite-dimensional for every $W$.
\end{example}

\subsection{An infinite family of simply-laced 2-in/2-out quivers}

The examples above illustrate two ways in which the obstruction can occur:
the regular construction allows parallel arrows, while the nine-vertex
quiver $Q_{3,9}$ is simply-laced. We now construct an
infinite family of simply-laced quivers admitting no Jacobi-finite potential.

For $n\geq5$, let $Q=Q_{2,n}$ be the quiver with vertex set
$\mathbb Z/n\mathbb Z$ and arrows
\begin{equation*}
	i\longrightarrow i+1,
	\qquad
	i\longrightarrow i+2
	\qquad(i\in\mathbb Z/n\mathbb Z).
\end{equation*}
See Figure~\ref{fig:q-n} for an example when $n=8$.

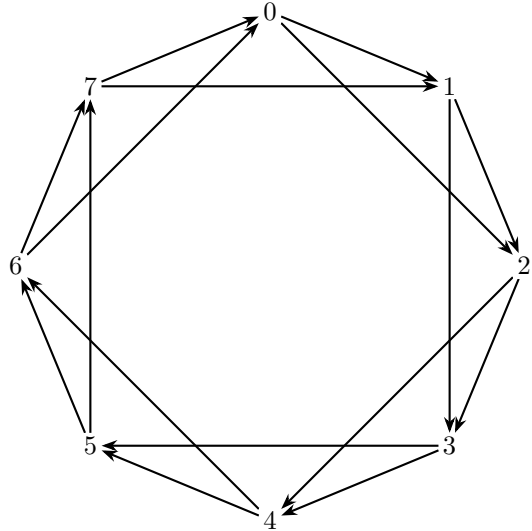
\begin{figure}[htbp]
	\centering
	\begin{tikzpicture}[
		scale=1.2,
		vertex/.style={
			font=\small,
			inner sep=1.5pt,
			fill=white
		},
		arr/.style={
			-{Stealth[length=2.1mm,width=1.5mm]},
			draw=black,
			line width=0.8pt
		}
		]
		
		\def\N{8}
		\def\r{2.8}
		\pgfmathtruncatemacro{\NmOne}{\N-1}
		
		\foreach \i in {0,...,\NmOne} {
			\node[vertex] (v\i) at ({90-360/\N*\i}:\r) {$\i$};
		}
		
		\foreach \i in {0,...,\NmOne} {
			\pgfmathtruncatemacro{\j}{mod(\i+2,\N)}
			\draw[arr] (v\i) -- (v\j);
		}
		
		\foreach \i in {0,...,\NmOne} {
			\pgfmathtruncatemacro{\j}{mod(\i+1,\N)}
			\draw[arr] (v\i) -- (v\j);
		}
		
		\foreach \i in {0,...,\NmOne} {
			\node[vertex] at (v\i) {$\i$};
		}
		
	\end{tikzpicture}
	\caption{The quiver $Q_{2,n}$, illustrated for $n=8$.  For every
		$i\in\mathbb Z/n\mathbb Z$, there are arrows from $i$ to $i+1$ and
		$i+2$.}
	\label{fig:q-n}
\end{figure}

\begin{proposition}
	\label{proposition:q-n}
	If $n\geq10$, then $\Jac(Q_{2,n},W)$ is infinite-dimensional for every field
	$K$ and every potential $W$ on $Q_{2,n}$.
\end{proposition}

\begin{proof}
	When $n\geq5$, the quiver $Q_{2,n}$ is $2$-acyclic and every vertex has
	two incoming and two outgoing arrows.  The directed distances
	$\ell_Q(a)$ for arrows of steps $1$ and $2$ are
	\[
	L_1(n)=\left\lceil\frac{n-1}{2}\right\rceil,
	\qquad
	L_2(n)=\left\lceil\frac{n-2}{2}\right\rceil.
	\]
	Indeed, a directed path from $i+s$ back to $i$ corresponds to a sum of
	$1$'s and $2$'s congruent to $n-s$ modulo $n$, and the least possible
	positive sum is $n-s$.  The minimum number of summands is therefore the
	corresponding ceiling displayed above.
	
	Take $\rho=2/3$ and $u=\one$.  At $n=10$ one has
	$L_1(10)=5$ and $L_2(10)=4$, so
	\begin{equation}
		\label{eq:q-ten-scalar}
		1-2\rho+\rho^{L_1(10)}+\rho^{L_2(10)}
		=
		1-\frac43+\left(\frac23\right)^5
		+\left(\frac23\right)^4
		=
		-\frac1{243}<0.
	\end{equation}
	Both directed distances are nondecreasing in $n$.  Hence the left-hand
	side of \eqref{eq:q-ten-scalar}, with $10$ replaced by $n$, remains
	negative for every $n\geq10$.  Translation invariance now gives
	\[
	\one^{\top}
	\bigl({\bf I}_{n}-\rho M^{\top}+C_Q(\rho)\bigr)<0
	\]
	entrywise.  Theorem~\ref{theorem:return-obstruction} applies.
\end{proof}

\begin{remark}
	Proposition~\ref{proposition:q-n} shows that neither parallel arrows nor
	high valency are responsible for the obstruction.  Each $Q_{2,n}$ is
	simply-laced and $2$-in/$2$-out, whereas in the regular case the coarser
	quadratic criterion of
	Proposition~\ref{proposition:regular-quivers} only begins at degree $4$.
	The return distances therefore contain genuinely finer information than
	the local arrow counts alone.
\end{remark}

Together, Proposition~\ref{proposition:regular-quivers} and
Proposition~\ref{proposition:q-n} prove
Theorem~\ref{theorem:main-counterexamples}.

    \section*{Acknowledgment}
    The first author is supported by the Fundamental Research Funds for the Central Universities (No. GK202403003), the NSF of China (Grant No. 12271321, 12671047). 
    
	\section*{Statement on AI usage}
	
	We acknowledge the use of AI during the preparation of this
	manuscript. The overall strategy of the paper, including the
	matrix-valued Golod--Shafarevich--Vinberg framework, the return-distance
	obstruction, and the construction of a coarser quadratic criterion, was developed
	by the authors.
	
	For the infinite family considered in
	Proposition~\ref{proposition:q-n}, the authors proposed studying the
	circulant quivers with arrows $i\longrightarrow i+1$ and
	$i\longrightarrow i+2$, while ChatGPT assisted with the parameter search
	that led to the choice \(\rho=2/3\). The resulting arguments and
	calculations were subsequently checked in detail, revised, and integrated
	into the manuscript by the authors. ChatGPT was also used for language
	editing. The authors have reviewed the manuscript in full and take full
	responsibility for all mathematical claims, arguments, and conclusions,
	as well as for any remaining errors.


\begin{thebibliography}{GLM22}
\raggedright
		
		\bibitem[A09]{Amiot}
		C.~Amiot,
		\emph{Cluster categories for algebras of global dimension $2$ and quivers
			with potential},
		Ann. Inst. Fourier (Grenoble) \textbf{59} (2009), no.~6, 2525--2590.
		
		\bibitem[BW25]{BrownWemyss}
		G.~Brown and M.~Wemyss,
		\emph{Local normal forms of noncommutative functions},
		Forum Math. Pi \textbf{13} (2025), Paper No.~e8, 59 pp.
		
		\bibitem[CG24]{CasalsGao}
		R.~Casals and H.~Gao,
		\emph{A Lagrangian filling for every cluster seed},
		Invent. Math. \textbf{237} (2024), no.~2, 809--868.
		
		\bibitem[CZ23]{ChangZhang}
		W.~Chang and J.~Zhang,
		\emph{Quivers with potentials for Grassmannian cluster algebras},
		Canad. J. Math. \textbf{75} (2023), no.~4, 1199--1225.
		
		\bibitem[DWZ08]{DerksenWeymanZelevinsky}
		H.~Derksen, J.~Weyman and A.~Zelevinsky,
		\emph{Quivers with potentials and their representations I: Mutations},
		Selecta Math. (N.S.) \textbf{14} (2008), no.~1, 59--119.
		
		\bibitem[EE07]{EtingofEu}
		P.~Etingof and C.-H.~Eu,
		\emph{Koszulity and the Hilbert series of preprojective algebras},
		Math. Res. Lett. \textbf{14} (2007), no.~4, 589--596.
		
		\bibitem[FST12]{FeliksonShapiroTumarkin}
		A.~Felikson, M.~Shapiro and P.~Tumarkin,
		\emph{Skew-symmetric cluster algebras of finite mutation type},
		J. Eur. Math. Soc. (JEMS) \textbf{14} (2012), no.~4, 1135--1180.
		
		\bibitem[FST08]{FominShapiroThurston}
		S.~Fomin, M.~Shapiro and D.~Thurston,
		\emph{Cluster algebras and triangulated surfaces. Part I: Cluster
			complexes},
		Acta Math. \textbf{201} (2008), no.~1, 83--146.
		
		\bibitem[FZ02]{FominZelevinsky}
		S.~Fomin and A.~Zelevinsky,
		\emph{Cluster algebras I: Foundations},
		J. Amer. Math. Soc. \textbf{15} (2002), no.~2, 497--529.
		
		\bibitem[GR21]{GaddisRogalski}
		J.~Gaddis and D.~Rogalski,
		\emph{Quivers supporting twisted Calabi--Yau algebras},
		J. Pure Appl. Algebra \textbf{225} (2021), no.~9,
		Paper No.~106645, 33 pp.
		
		\bibitem[GLS16]{GeissLabardiniSchroer}
		C.~Gei{\ss}, D.~Labardini-Fragoso and J.~Schr\"oer,
		\emph{The representation type of Jacobian algebras},
		Adv. Math. \textbf{290} (2016), 364--452.
		
		\bibitem[GLM22]{GeuenichLabardiniMiranda}
		J.~Geuenich, D.~Labardini-Fragoso and J.~L.~Miranda-Olvera,
		\emph{Quivers with potentials associated to triangulations of closed surfaces
			with at most two punctures},
		S\'em. Lothar. Combin. \textbf{84} (2022), Art.~B84c, 21 pp.
		
		\bibitem[G06]{Ginzburg}
		V.~Ginzburg,
		\emph{Calabi--Yau algebras},
		preprint, arXiv:math/0612139 (2006).
		
		\bibitem[G64]{Golod}
		E.~S.~Golod,
		\emph{On nil-algebras and finitely approximable $p$-groups},
		Izv. Akad. Nauk SSSR Ser. Mat. \textbf{28} (1964), no.~2, 273--276.
		
		\bibitem[GS64]{GolodShafarevich}
		E.~S.~Golod and I.~R.~Shafarevich,
		\emph{On the class field tower},
		Izv. Akad. Nauk SSSR Ser. Mat. \textbf{28} (1964), no.~2, 261--272.
		
		\bibitem[HY25]{HaerizadehYurikusa}
		M.~Haerizadeh and T.~Yurikusa,
		\emph{Finite-dimensional Jacobian algebras: Finiteness and tameness},
		preprint, arXiv:2507.04570 (2025).
		
		\bibitem[IS22]{IyuduShkarin}
		N.~Iyudu and S.~Shkarin,
		\emph{Golod--Shafarevich--Vinberg type theorems and finiteness conditions
			for potential algebras},
		preprint, arXiv:2201.04479 (2022).
		
		\bibitem[IS19]{IyuduSmoktunowicz}
		N.~Iyudu and A.~Smoktunowicz,
		\emph{Golod--Shafarevich-type theorems and potential algebras},
		Int. Math. Res. Not. IMRN (2019), no.~15, 4822--4844.
		
		\bibitem[KY11]{KellerYang}
		B.~Keller and D.~Yang,
		\emph{Derived equivalences from mutations of quivers with potential},
		Adv. Math. \textbf{226} (2011), no.~3, 2118--2168.
		
		\bibitem[L09]{LabardiniFragoso}
		D.~Labardini-Fragoso,
		\emph{Quivers with potentials associated to triangulated surfaces},
		Proc. Lond. Math. Soc. (3) \textbf{98} (2009), no.~3, 797--839.
		
		\bibitem[L16]{LabardiniFragosoIV}
		D.~Labardini-Fragoso,
		\emph{Quivers with potentials associated to triangulated surfaces, Part IV:
			Removing boundary assumptions},
		Selecta Math. (N.S.) \textbf{22} (2016), no.~1, 145--189.
		
		\bibitem[L12]{LadkaniClosedSurfaces}
		S.~Ladkani,
		\emph{On Jacobian algebras from closed surfaces},
		preprint, arXiv:1207.3778 (2012).
		
		\bibitem[L13]{LadkaniClassP}
		S.~Ladkani,
		\emph{On cluster algebras from once punctured closed surfaces},
		preprint, arXiv:1310.4454 (2013).
		
		\bibitem[L25]{LadkaniXSeven}
		S.~Ladkani,
		\emph{Non-degenerate potentials on the quiver $X_7$},
		J. Algebra \textbf{666} (2025), 94--148.
		
		\bibitem[LP24]{LiPeng}
		Y.~Li and L.~Peng,
		\emph{Finite dimensional $2$-cyclic Jacobian algebras},
		preprint, arXiv:2408.10056 (2024).
		
		\bibitem[S23]{SchroerAtlas}
		J.~Schr\"oer,
		\emph{Atlas of finite-dimensional algebras},
		manuscript, version of 29 December 2023.
		
		\bibitem[TV17]{TrepodeValdivieso}
		S.~Trepode and Y.~Valdivieso-D\'iaz,
		\emph{On finite dimensional Jacobian algebras},
		Bol. Soc. Mat. Mex. (3) \textbf{23} (2017), no.~2, 653--666.
		
		\bibitem[V65]{Vinberg}
		E.~B.~Vinberg,
		\emph{On the theorem concerning the infinite-dimensionality of an
			associative algebra},
		Izv. Akad. Nauk SSSR Ser. Mat. \textbf{29} (1965), no.~1, 209--214;
		English transl., Amer. Math. Soc. Transl. (2) \textbf{82} (1969), 237--242.
		
	\end{thebibliography}
\end{document}